\documentclass[a4paper, 12pt]{article}
\usepackage{amsmath, amssymb, amscd, amsthm, amsfonts, mathrsfs, authblk, showlabels, amssymb}
\usepackage{graphicx}
\usepackage{hyperref}

\newtheorem{thm}{Theorem}

\newtheorem{prop}{Proposition}

\newtheorem{pblm}{Problem}
\newtheorem{clm}{Claim}
\newtheorem{fact}{Fact}

\newcommand{\gx}{G^{\times}}
\newcommand{\ch}{{\rm ch}}
\begin{document}

\title{Dynamic list coloring of 1-planar graphs}

\author[1]{Xin Zhang\thanks{Supported by the National Natural Science Foundation of China (11871055) and the Youth Talent Support Plan of Xi'an Association for Science and Technology (2018-6).}\thanks{Corresponding author. Email: xzhang@xidian.edu.cn.}}

\author[1]{Yan Li}

\affil[1]{\small School of Mathematics and Statistics, Xidian University, Xi'an~710071,~China}
\date\today

\maketitle
\begin{abstract}
   
A graph is $k$-planar if it can be drawn in the plane so that each edge is crossed at most $k$ times. Typically, the class of 1-planar graphs is among the most investigated graph families within the so-called ``beyond planar graphs".
A dynamic $\ell$-list coloring of a graph is a proper coloring so that each vertex receives a color from a list of $\ell$ distinct candidate colors assigned to it, and meanwhile, there are at least two colors appearing in the neighborhood of every vertex of degree at least two. In this paper, we prove that each 1-planar graph has a dynamic $11$-list coloring. Moreover, we show a relationship between the dynamic coloring of 1-planar graphs and the proper coloring of 2-planar graphs, which states that the dynamic (list) chromatic number of the class of 1-planar graphs is at least the (list) chromatic number of the class of 2-planar graphs.

\bigskip
\noindent \emph{Keywords: $1$-planar graph; $2$-planar graph; dynamic coloring; proper coloring; list coloring}.

\end{abstract}

\section{Introduction}

Given a simple graph $G$ with vertex set $V(G)$ and edge set $E(G)$, we use $N_G(v)$ to denote the set of neighbors of $v$ in $G$ and say that $d_G(v)=|N_G(v)|$ is the \emph{degree of $v$} in $G$. A \emph{planar graph} is a graph admitting a drawing in the plane with no crossing and typically we say such a drawing a \emph{plane graph}. By $F(G)$ we denote the face set of a plane graph $G$, and for any face $f\in F(G)$, we use $d_G(f)$ to denote the \emph{degree of $f$} in $G$, which is the number of edges that are incident with $f$ in $G$ (cut-edges are counted twice). By $V_G(f)$, we denote the set of vertices  incident with a face $f$ in a plane graph $G$.
A $t$-, $t^+$-, or $t^-$-vertex (resp.\,face) is a vertex (resp.\,face) of degree $t$, at least $t$, or at most $t$, respectively.

A \emph{proper $\ell$-coloring} of a graph $G$ is a coloring on $V(G)$ using $\ell$ colors so that adjacent vertices receive distinct colors. If every vertex of degree at least two is incident with at least two colors, then we call this proper $\ell$-coloring a \emph{dynamic $\ell$-coloring}. The minimum integer $\ell$ 
such that $G$ has a proper (resp.\,dynamic) $\ell$-coloring is  the \emph{chromatic number} (resp.\,\emph{dynamic chromatic number}) of $G$, denoted by $\chi(G)$ (resp.\,$\chi^d(G)$). 

The well-known four color theorem states that $\chi(G)\leq 4$ for every planar graph $G$. In 2013,
Kim, Lee, and Park \cite{Kim20132207} proved that $\chi^d(G)\leq 5$ for every planar graph $G$
and the equality holds if and only if $G\cong C_5$, answering a conjecture of Chen \emph{et al.}\,\cite{Chen20121064}. Furthermore, the same conclusion holds even for  $K_5$-minor-free graphs, which was proved by Kim, Lee and Oum \cite{Kim201681} in 2016. For other results on the dynamic coloring of graphs, we refer the reads to \cite{Ahadi20122579,Alishahi2011152,Borowiecki2012105,Bowler2017151,Chen20121064,Karpov2011601,Lai2003193,Loeb2018129,Meng20063,Montgomery2001,Saqaeeyan2016249,Vlasova201821}.

Imaging that each vertex $v\in V(G)$ is assigned a \emph{list} $L(v)$ of distinct candidate colors,  our goal is to color the vertices of $G$ so that every vertex receives color from its list assignment and the resulting coloring of $G$ is a proper (resp.\,dynamic) coloring. If we win for a given list assignment $L$ to $V(G)$, then $G$ is \emph{$L$-colorable} (resp.\,\emph{dynamically $L$-colorable}). Furthermore, if we win for every given list assignment $L$ to $V(G)$ with $|L(v)|=\ell$ for each $v\in V(G)$, then $G$ is \emph{$\ell$-choosable} (resp.\,\emph{dynamically $\ell$-choosable}). The minimum integer $\ell$ so that $G$ is $\ell$-choosable (resp.\,dynamically $\ell$-choosable) is the \emph{list chromatic number} (resp.\,\emph{dynamic list chromatic number}) of $G$, denoted by $\ch(G)$ (resp.\,$\ch^d(G)$).

Thomassen's theorem \cite{Thomassen1994180} states that $\ch(G)\leq 5$ for every planar graph $G$, and the sharpness of this upper bound $5$ was confirmed by Voigt \cite{Voigt1993215}, who constructed a planar graph $G$ with $\chi(G)=4$ and $\ch(G)=5$. This reminds us that $\chi(G)$ and $\ch(G)$ are not always the same, even for planar graphs. Similarly, Esperet \cite{Esperet20101963} showed that there is a planar bipartite graph $G$ with $\ch(G)=\chi^d(G)=3$ and $\ch^d(G)=4$, and moreover, there exists for every $k\geq 5$ a bipartite graph $G_k$ with $\ch(G_k)=\chi^d(G_k)=3$ and $\ch^d(G_k)\geq k$. Hence the gap between $\chi^d(G)$ (or $\ch(G)$) and $\ch^d(G)$ can be any large. For further interesting readings on the dynamic list coloring of graphs, we refer the readers to \cite{Alishahi2011152,Kim20132207,Kim2011156}.

A \emph{$2$-subdivision} of a graph $G$ is the graph derived from $G$ by inserting on each edge a new vertex of degree two, denoted by $G^{\star}$. One can see Figure \ref{K7} for an example of  $K_7^{\star }$, which is 1-planar.
\begin{figure}
    \centering
    \includegraphics[width=8cm]{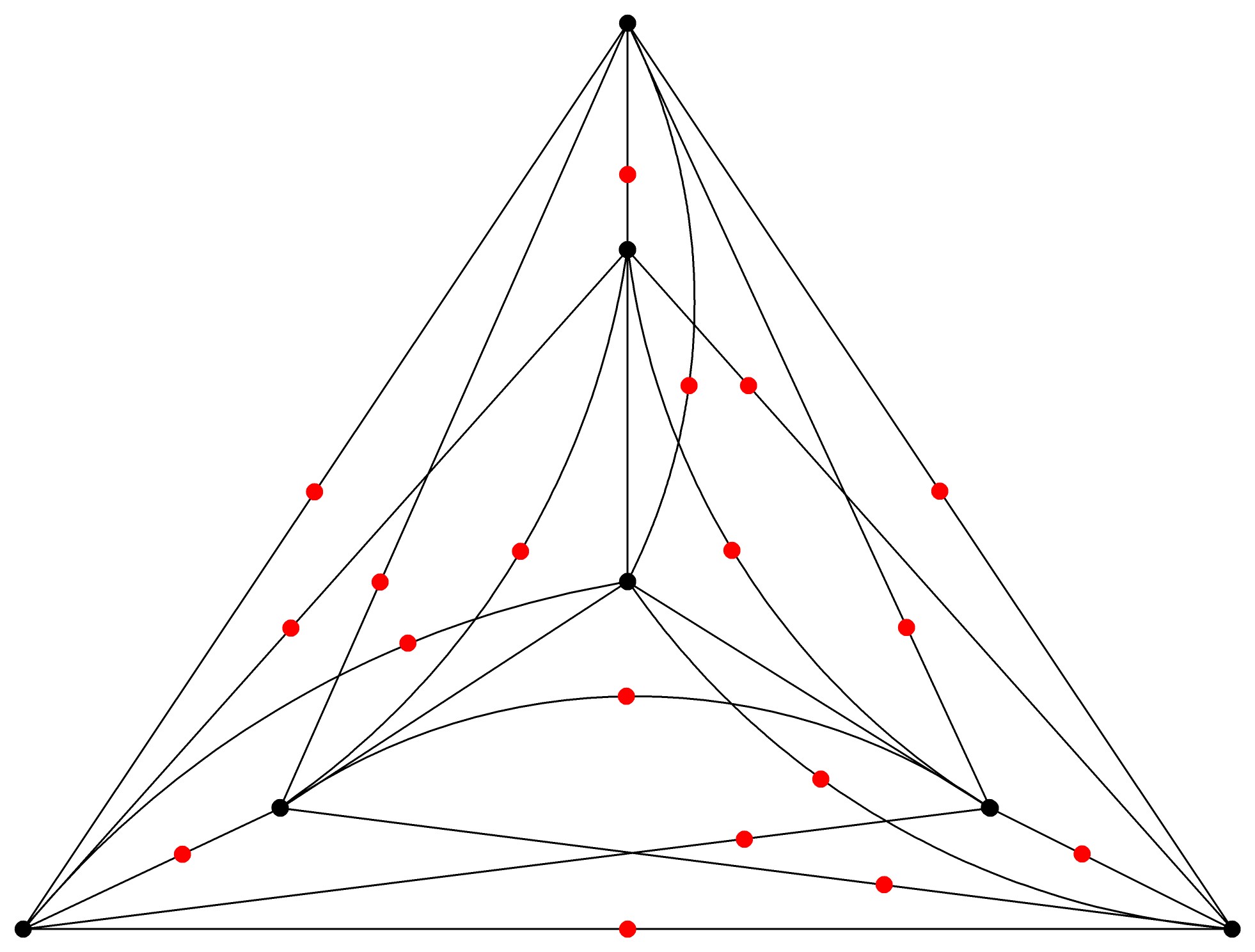}
    \caption{A 1-planar drawing of $K_7^{\star }$}
    \label{K7}
\end{figure}

\begin{fact}\label{fact-1}
For any graph $G$, $\chi(G)\leq \chi^d(G^{\star}) $ and $\ch(G)\leq \ch^d(G^{\star})$.
\end{fact}
\begin{proof}
Let $M: V(G) \rightarrow V(G^{\star})$ be a mapping that maps a vertex of $G$ to the vertex of $G^{\star}$ corresponding to it, and let $S\subset V(G^{\star})$ be the set of new added 2-vertices to $G$ while doing the 2-subdivision. Let $L$ be an arbitrary $\ell$-list assigment on $V(G)$, where $\ell=\ch^d(G^{\star})$. 
We extend $L$ to an $\ell$-list $L^\star$ on $V(G^\star)$, i.e, $L^\star(u)=L(M^{-1}(u))$ for any $u\in V(G^{\star})\backslash S$.
Since the two neighbors of a 2-vertex of $G^\star$ shall be colored with distinct colors in any dynamic coloring of $G^\star$, there is a dynamic coloring $c^\star$ of $G^\star$ so that $c^\star(M(u))\in L^\star(M(u))=L(u)$, $c^\star(M(v))\in L^\star(M(v))=L(v)$, and $c^\star(M(u))\neq c^\star(M(v))$ for any $uv\in E(G)$. Therefore, we   construct an $L$-coloring $c$ of $G$ by letting $c(v)=c^\star(M(v))$ for any $v\in V(G)$. This implies that 
$\ch(G)\leq \ell=\ch^d(G^{\star})$. 
The proof for  $\chi(G)\leq \chi^d(G^{\star})$ is similar (we just proceed by fixing every $\ell$-list used in the privious proof to be $\{1,2,\ldots,\ell\}$).
\end{proof}

Note that the equality in Fact \ref{fact-1} does not always hold. One easy example is the cycle $C_n$ on $n$ vertices. Since 
$C_n^\star=C_{2n}$ and it is known \cite{Akbari20093005,Lai2003193,Montgomery2001} that 
\begin{align*}
        \chi^d(C_{2n})=\ch^d(C_{2n})
        &=
        \begin{cases}
        3
        &\text{if } n  \equiv 0 ~({\rm mod~3}),
        \\
        4
        &\text{if } n  \not\equiv 0 ~({\rm mod~3}),
        \end{cases}
   \end{align*}
we have 
\begin{align*}
        \chi^d(C_n^{\star})-\chi(C_n)=\ch^d(C_n^{\star})-\ch(C_n)
        &=
        \begin{cases}
        0
        &\text{if } n  \equiv 3 ~({\rm mod~6}),
        \\
        1
        &\text{if } n  \equiv 0,1,5 ~({\rm mod~6}),
          \\
        2
        &\text{if } n  \equiv 2,4 ~({\rm mod~6}).
        \end{cases}
   \end{align*}

A graph is \emph{$k$-planar} if it can be drawn in the plane so that each edge is crossed at most $k$ times. Specially, the 1-planarity was initially introduced by Ringel \cite{Ringel1965107} in 1965, who proved that $\chi(G)\leq 7$ for every 1-planar graph and conjectured that every 1-planar graph is 6-colorable. This conjecture was solved by Borodin \cite{Borodin1984} in 1984, who also gave a new proof \cite{Borodin1995507} in 1995. Due to the 1-planar graph $K_6$, the upper bound 6 for the chromatic number of the class of 1-planar graphs is sharp. Since 2006, the list coloring of 1-planar graphs was also investigated by many researchers including Albertson and Mohar \cite{Albertson2006289}, Wang and Lih \cite{Wang200827}. In particular, the second group \cite{Wang200827} proved that $\ch(G)\leq 7$ for every 1-planar graph $G$. Actually, the class of 1-planar graphs is among the most investigated graph families within the so-called ``beyond planar graphs", see \cite{Didimo2019}. For those who want to know more about 1-planar graphs, we refer them to a recent survey due to Kobourov, Liotta and Montecchiani \cite{Kobourov201749}.

Let $\mathcal{G}_k$ be the class of graphs that are $k$-planar and non-$(k-1)$-planar. By $\chi(\mathcal{G}_k)$  we denote the minimum integer $\ell$ so that $\chi(G)\leq \ell$  for each $G\in \mathcal{G}_k$. Similarly, we can define $\chi^d(\mathcal{G}_k)$, $\ch(\mathcal{G}_k)$, and $\ch^d(\mathcal{G}_k)$.
 If $G\in \mathcal{G}_{k+1}$ with $k\geq 1$, then it is easy to see that the 2-subdivision of $G$ is $k$-planar. 

Pach and Tóth \cite{Pach1997427} showed that $|E(G)|\leq 5|V(G)|-10$ for each 2-planar graph $G$. This implies that each 2-planar graph $G$ has a vertex of degree at most $9$ and thus $\chi(G)\leq \ch(G)\leq 10$. Since $K_7$ is a non-1-planar 2-planar graph, $7\leq \chi(\mathcal{G}_2)\leq 10$ and $7\leq \ch(\mathcal{G}_2)\leq 10$.

We now look back at Fact \ref{fact-1}. 
If there is a 2-planar graph $G$ with $\chi(G)=\ell$ (resp.\,$\ch(G)=\ell$), then $G^\star$ is a 1-planar graph  with $\chi^d(G^\star)\geq \ell$ (resp.\,$\ch^d(G^\star)\geq \ell$). This implies

\begin{fact}\label{fact-2}
$\chi^d(\mathcal{G}_1)\geq \chi(\mathcal{G}_2)\geq 7$ and $\ch^d(\mathcal{G}_1)\geq \ch(\mathcal{G}_2)\geq 7$.
\end{fact}

The aim of this paper is to give a reasonable upper bound, say 11, for $\ch^d(\mathcal{G}_1)$ (note that $\chi^d(\mathcal{G}_1)\leq \ch^d(\mathcal{G}_1)$). In other words, we prove the following.

\begin{thm}\label{main-thm}
If $G$ is a 1-planar graph, then $\ch^d(G)\leq 11$.
\end{thm}

\section{Dynamically Minimal Graphs}

A graph class $\mathcal{F}$ is \emph{hereditary} if $\mathcal{F}$ is closed by taking subgraphs.
A graph $G$ is \emph{dynamically $\ell$-minimal in a hereditary class $\mathcal{F}$} if $G\in \mathcal{F}$ is not dynamically $\ell$-choosable and any graph $H\in \mathcal{F}$ with $|V(H)|+|E(H)|<|V(G)|+|E(G)|$ is dynamically $\ell$-choosable.

In this section, we use $\mathcal{G}_{1}^-$ to stand the class of 1-planar graphs, i.e., $\mathcal{G}_{1}^-=\mathcal{G}_0\cup \mathcal{G}_1$.
Suppose that $G$ is a dynamically $\ell$-minimal graph in $\mathcal{G}_{1}^-$, 
It follows that $G$ is a 1-planar graph with the smallest value of $|V(G)|+|E(G)|$ such that
there is an $\ell$-list assignment $L$ to the vertices of $G$ such that $G$ is not dynamically $L$-colorable. Moreover, we assume that $G$ is a \emph{$1$-plane graph} (i.e, a drawing of $G$ in the plane so that its 1-planarity is satisfied) that has the minimum number of crossings. 

The \emph{associated plane graph} $\gx$ of a 1-plane $G$ is the plane graph derived from $G$ by turning all crossings into new vertices of degree 4, and those 4-vertices in $\gx$ are called \emph{false vertices}. If a vertex of $\gx$ is not false, then it is a \emph{true vertex}. A face of the plane graph $\gx$ is a \emph{false face} if it is incident with at least one false vertex, and is a \emph{true face} otherwise. Clearly, no two false vertices are adjacent in $\gx$ by the definition of the 1-planarity and  each face $f$ of $\gx$ is incident with at most $d_{\gx}(f)/2$ false vertices.

In the following statements or the proofs of the propositions, $\mathcal{F}$ stands for an arbitrary given hereditary graph class, and $L$ is the $\ell$-list assignment mentioned above.

\begin{prop} \label{min-deg}
If $G$ is a dynamically $\ell$-minimal graph in $\mathcal{F}$ with $\ell\geq 3$, then $\delta(G)\geq 2$.
\end{prop}

\begin{proof}
Suppose, to the contrary, that $G$ has an edge $uv$ with $d_G(u)=1$. By the minimality of $G$, $G'=G-u\in \mathcal{F}$ is dynamically $L$-colorable. Let $c$ be a dynamic $L$-coloring of $G'$. Coloring $u$ from $L(u)$ with a color different from the colors on $v$ and a neighbor of $v$ besides $u$, we obtain a dynamic $L$-coloring of $G$, a contradiction.
\end{proof}

\begin{prop} \label{no-two-two}
If $G$ is a dynamically $\ell$-minimal graph in $\mathcal{F}$ with  $\ell\geq 5$, then no two $2$-vertices are adjacent in $G$.
\end{prop}

\begin{proof}
Suppose, to the contrary, that $G$ has an edge $uv$ with $d_G(u)=d_G(v)=2$. Let $N_G(u)=\{v,x\}$, $N_G(v)=\{u,y\}$, $x_1\in N_G(x)\backslash \{u\}$, and $y_1\in N_G(y)\backslash \{v\}$. By the minimality of $G$, $G'=G-\{u,v\}\in \mathcal{F}$ has a dynamic $L$-coloring $c$. Coloring $u$ with $c(u)\in L(u)\backslash \{c(x),c(y),c(x_1)\}$ and $v$ with $c(v)\in L(v)\backslash \{c(u),c(x),c(y),c(y_1)\}$, we get a dynamic $L$-coloring of $G$, a contradiction.
\end{proof}

Actually, Proposition \ref{no-two-two} can be generalized to the following.

\begin{prop} \label{edge-with-big-vertex}
If $G$ is a dynamically $\ell$-minimal graph in $\mathcal{F}$ with $\ell\geq 5$, then each edge of $G$ is incident with at least one $\ell^+$-vertex.
\end{prop}

\begin{proof}
We first claim that if $uv\in E(G)$ and $d_G(u)=2$, then $d_G(v)\geq \ell$.
Suppose, to the contrary, that $d_G(v)\leq \ell-1$. Let $N_G(u)=\{v,z\}$ and  $N_G(v)=\{u,x_1,\ldots, x_{t},y_1,\ldots,y_s\}$, where $d_G(x_i)=2$ for each $1\leq i\leq t$ and $d_G(y_i)\geq 3$ for each $1\leq i\leq s$. Let $N_G(x_i)=\{v,x'_i\}$ for each $1\leq i\leq t$. Note that $t$ or $s$ may be 0, in which case  $N_G(v)=\{u,y_1,\ldots,y_s\}$ or $N_G(v)=\{u,x_1,\ldots,x_t\}$, respectively. 
By Proposition \ref{no-two-two}, $t+s\geq 2$, 
$d_G(z)\geq 3$, and $d_G(x'_i)\geq 3$ for each $1\leq i\leq t$. 
By the minimality of $G$, $G'=G-\{u,v,x_1,\ldots,x_t\}\in \mathcal{F}$ has a dynamic $L$-coloring $c$. 
Color $v,x_1,\ldots,x_t,u$ in this order with colors
$c(v),c(x_1),\ldots,c(x_t),c(u)$ such that
$c(v)\in L(v)\backslash F(v)$, where $F(v)=\{c(z),c(x'_1),\ldots,c(x'_t),c(y_1),\ldots,c(y_s)\}$,
$c(x_i)\in L(x_i)\backslash \{c(v),c(x'_i)\}$ for each $1\leq i\leq t$, and
$c(u)\in L(u)\backslash \{c(z),c(v),c(x_1)\}$ if $t\neq 0$, or
$c(u)\in L(u)\backslash \{c(z),c(v),c(y_1)\}$ if $t=0$.
Note that $|F(v)|=1+t+s=d_G(v)\leq \ell-1$. It is easy to see that this results in a dynamic $L$-coloring of $G$, a contradiction. 

We come back to the proof of Proposition \ref{edge-with-big-vertex}. Suppose, to the contrary, that $G$ has an edge $uv$ with $d_G(u)=a\leq d_G(v)=b\leq \ell-1$. Let $N_G(u)=\{v,u_1,\ldots, u_{a-1}\}$ and $N_G(v)=\{u,v_1,\ldots, v_{b-1}\}$. By the arguments in the first paragraph, $a,b\geq 3$ and $d_G(u_i),d_G(v_j)\geq 3$ for each $1\leq i\leq a-1$ and $1\leq j\leq b-1$. 
By the minimality of $G$, $G'=G-\{u,v\}\in \mathcal{F}$ has a dynamic $L$-coloring $c$. Without loss of generality, assume that $c(v_1)\leq c(v_2)\leq \cdots \leq c(v_{b-1})$.

If $c(v_1)\neq c(v_{b-1})$, then we construct a dynamic $L$-coloring of $G$ by coloring $v$ and $u$ in order with $c(v)$ and $c(u)$ such that 
$c(v)\in L(v)\backslash F(v)$ and $c(u)\in L(u)\backslash F(u)$, where $F(v)=\{c(v_1),\ldots,c(v_{b-1}),c(u_1)\}$ and $F(u)=\{c(u_1),\ldots,c(u_{a-1}),c(v)\}$. 

If  $c(v_1)= c(v_{b-1})$, then we construct a dynamic $L$-coloring of $G$ by coloring $u$ and $v$ in order with $c(u)$ and $c(v)$ such that 
$c(u)\in L(u)\backslash F(u)$ and $c(v)\in L(v)\backslash F(v)$, where  $F(u)=\{c(u_1),\ldots,c(u_{a-1}),c(v_1)\}$ and $F(v)=\{c(v_1),c(u),c(u_1)\}$.

In each of the above two cases we win since $|F(u)|=a\leq \ell-1$ and $|F(v)|\leq b\leq \ell-1$. So we have contradictions.
\end{proof}

\begin{prop}\label{true-3-face}
If $G$ is a dynamically $\ell$-minimal graph in $\mathcal{F}$ with $\ell\geq 5$ and $u$ is a vertex incident with a triangle,
then $d_G(u)\geq \ell$.
\end{prop}

\begin{proof}
Suppose, to the contrary, that $f=uvw$ is a triangle such that $d_G(u)\leq \ell-1$. By the minimality of $G$, $G'=G-u
\in \mathcal{F}$ has a dynamic $L$-coloring $c$. By Proposition \ref{edge-with-big-vertex}, $d_G(v),d_G(w)\geq \ell$.
Let $N_G(u)=\{v,w,x_1,\ldots, x_{t}\}$. Since $d_G(u)\leq k-1$, $d_G(x_i)\geq \ell$ for each $1\leq i\leq t$ by Proposition \ref{no-two-two}. Extending $c$ to a dynamic $L$-coloring of $G$ by coloring $u$ with a color $c(u)\in L(u)\backslash F(u)$, where $F(u)=\{c(v),c(w),c(x_1),\ldots, c(x_t)\}$, we find a contradiction. Note that $|F(u)|=t+2=d_G(u)\leq \ell-1$.
\end{proof}

\begin{prop}\label{false-3-face}
If $G$ is a dynamically $\ell$-minimal graph in $\mathcal{G}_{1}^-$ with $\ell\geq 5$ and $u$ is a vertex incident with a false $3$-face of $\gx$, then either $u$ is false or $d_G(u)\geq \ell-2$.
\end{prop}

\begin{proof}
Suppose, to the contrary, that $u$ is true and $d_G(u)\leq \ell-3$. Let 
$f=upw$ be the false 3-face that is incident with $u$, where $p$ is a false vertex. Basically we assume 
$uv$ crosses $ww'$ in $G$ at a point $p$.
Let $N_G(u)=\{v,w,x_1,\ldots, x_{t}\}$. Since $d_G(u)\leq \ell-3$, by Proposition \ref{edge-with-big-vertex}, $d(v),d(w)\geq \ell$ and $d_G(x_i)\geq \ell$ for each $1\leq i\leq t$. 
If $vw\in E(G)$, then let $G'=G-u$. If $vw\not\in E(G)$, then let $G'=G-u+vw$.
In any case, we can see that $G'$ is still 1-planar, i.e, $G'\in\mathcal{G}_{1}^-$. 
Let $v'$ be another neighbor of $v$ in $G'$ that is not $u$ or $w$ or $w'$. 
By the minimality of $G$, $G'$ has a dynamic $L$-coloring $c$. Extending $c$ to a dynamic $L$-coloring of $G$ by coloring $u$ with a color $c(u)\in L(u)\backslash F(u)$, where $F(u)=\{c(v),c(w),c(x_1),\ldots, c(x_t),c(v'),c(w')\}$, we get a contradiction. Note that $|F(u)|=t+4=d_G(u)+2\leq \ell-1$.
\end{proof}

\begin{prop}\label{big-face}
If $G$ is a dynamically $\ell$-minimal graph in $\mathcal{G}_{1}^-$ with $\ell\geq 5$ and $f=wuvy_1\cdots y_s$ is 
a $4^+$-face of $\gx$ with $d_G(u)\leq \ell-3$, where $s\geq 1$, then both $v$ and $w$ are false. \end{prop}

\begin{proof}
Suppose, to the contrary, that at least one of $v$ and $w$ is true. We divide the proof into two major cases.

First of all, we assume that both $v$ and $w$ are true. If $vw\in E(G)$, then let $G'=G-u$. If $vw\not\in E(G)$, then let $G'=G-u+vw$.
In any case, it is easy to see that $G'$ is still 1-planar, i.e, $G'\in \mathcal{G}_1^-$. By the minimality of $G$, $G'$ has a dynamic $L$-coloring $c$. Let $N_G(u)=\{v,w,x_1,\ldots,x_t\}$. By Proposition \ref{edge-with-big-vertex}, any neighbor of $u$ in $G$ has degree at least $\ell$ since $d_G(u)\leq \ell-3$.  Let $v'$ be another neighbor of $v$ in $G$ that is not among $\{x_1,\ldots,x_t,u,w\}$, and let $w'$ be another neighbor of $w$ in $G$ that is not among $\{x_1,\ldots,x_t,u,v,v'\}$ (such vertices exist since the number of the excluded vertices are at most $t+3=d_G(u)+1\leq \ell-2$). Color $u$ with a color $c(u)\in L(u)\backslash F(u)$, where $F(u)=\{c(x_1),\ldots,c(x_t),c(v),c(w),c(v'),c(w')\}$. Since $|F(u)|=t+4=d_G(u)+2\leq \ell-1$, we obtain a dynamic $L$-coloring of $G$, a contradiction.

On the other hand, we assume, by symmetry, that $v$ is true and $w$ is false. Basically we assume that $uu'$ crosses $w'y_s$ in $G$ at the point $w$. If $u'v\in E(G)$, then let $G'=G-u$, and otherwise let $G'=G-u+u'v$. The 1-planarity of $G'$ is easy to be confirmed (note that the crossing point $w$ in $G$ is removed by the deletion of $u$, and if we have to add the edge $u'v$, it can be drawn so that it is only crossed by $w'y_s$ in $G'$). By the minimality of $G$, $G'\in \mathcal{G}_1^-$ has a dynamic $L$-coloring $c$. Let $N_G(u)=\{u',v,x_1,\ldots,x_t\}$. By Proposition \ref{edge-with-big-vertex}, any neighbor of $u$ in $G$ has degree at least $\ell$ since $d_G(u)\leq \ell-3$.  Let $u''$ or $v'$ be another neighbor of $u'$ or $v$ in $G$ that is not among $\{x_1,\ldots,x_t,u,v\}$ or $\{x_1,\ldots,x_t,u,u''\}$, respectively.  
Color $u$ with a color $c(u)\in L(u)\backslash F(u)$, where $F(u)=\{c(x_1),\ldots,c(x_t),c(v),c(u'),c(u''),c(y_1)\}$. Since $|F(u)|=t+4=d_G(u)+2\leq \ell-1$, we get a dynamic $L$-coloring of $G$, a contradiction.
\end{proof}

\section{Discharging: the Proof of Theorem \ref{main-thm}}

If Theorem \ref{main-thm} is false, then there is a dynamically $11$-minimal 1-planar graph $G$. 
For every element $x\in V(\gx)\cup F(\gx)$, we assign an initial charge $c(x)=d_{\gx}(x)-4$. By the well-known Euler formulae $|V(\gx)|+|F(\gx)|-|E(\gx)|=2$ on the plane graph $\gx$, we have
$$\sum_{x\in V(\gx)\cup F(\gx)}c(x)=-8<0.$$

If there is a 4-face $f=uxvy$ in $\gx$ such that $d_{\gx}(u)\geq 11$, $2\leq d_{\gx}(v):=d\leq 3$ and $x,y$ are false vertices, then we call $f$ a \emph{special $4$-face}. 

Initially, we define the following discharging rules  (also see Figure \ref{Rules}) so that the charges are transferred among the elements in
$V(\gx)\cup F(\gx)$. 

\begin{enumerate}
    \item[R1.] Every true 3-face in $\gx$ receives $\frac{1}{3}$ from each of its incident $11^+$-vertices; 
    \item[R2.] Every false 3-face in $\gx$ receives $\frac{1}{2}$ from each of its incident $9^+$ vertices;
    \item[R3.] Every $11^+$-vertex incident with a special $4$-face $f$ sends $1$ to $f$, from which the special $3^-$-vertex on $f$ receives $1$;
    \item[R4.] Every  $5^+$-face in $\gx$ sends 1 to each of its incident special 2-vertices if there are some ones;
    \item[R5.] After applying R1--R4, every  $5^+$-face in $\gx$ redistributes its charge equitably to each of its incident non-special 2-vertices or (special or non-special) $3$-vertices if there are some ones.
\end{enumerate}

\begin{figure}
    \centering
    \includegraphics[width=15cm]{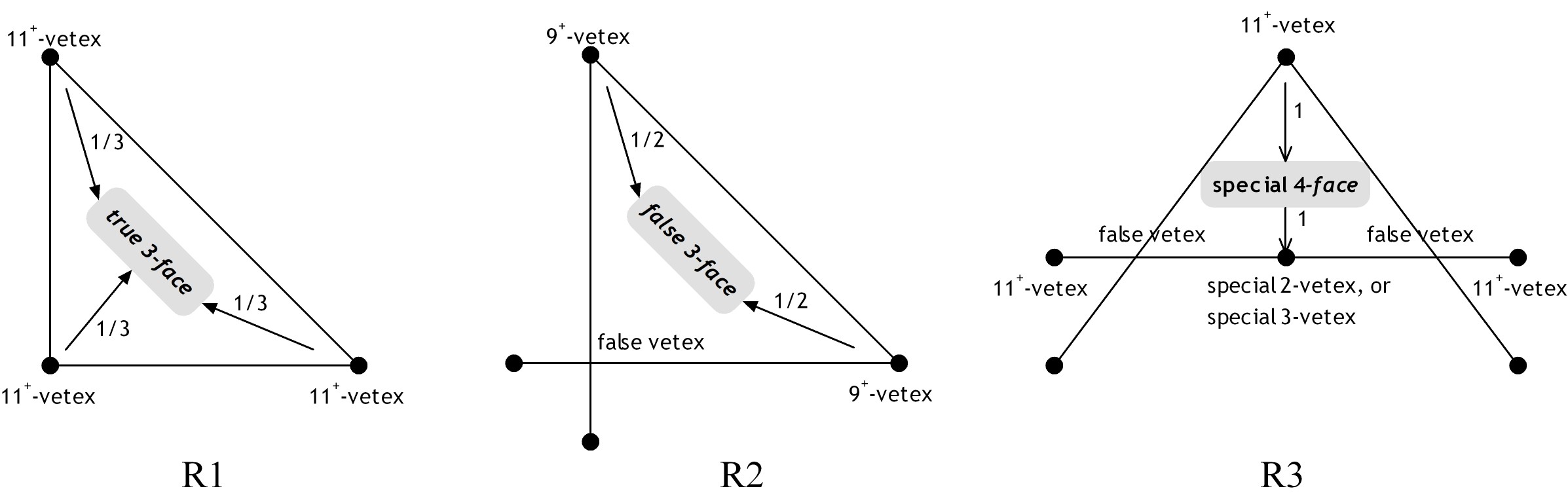}
    \caption{The discharging rules R1--R3}
    \label{Rules}
\end{figure}

Let $c'(x)$ be the final charge of the element $x\in V(\gx)\cup F(\gx)$ after discharging. Clearly
$$\sum_{x\in V(\gx)\cup F(\gx)}c'(x)=\sum_{x\in V(\gx)\cup F(\gx)}c(x)=-8<0.$$
In the following, we show that $c'(x)\geq 0$ for every $x\in V(\gx)\cup F(\gx)$ by Claims \ref{clm:5-minus-face-nonneg}, \ref{clm:6-plus-face-nonneg}, \ref{clm:2-vertex-nonneg}, \ref{clm:3-vertex-nonneg}, and \ref{clm:4-plus-vertex-nonneg}. This  contradiction completes the proof of Theorem \ref{main-thm}.

\begin{clm}\label{clm:5-minus-face-nonneg}
 Every $5^-$-face in $\gx$ has a nonnegative final charge.
\end{clm}

\begin{proof}
If $f$ is a true 3-face (i.e.,\,a triangle in $G$), then every vertex incident with $f$ is a $11^+$-vertex by Proposition \ref{true-3-face}, and thus $c'(f)=3-4+3\times\frac{1}{3}=0$ by R1.
If $f$ is a false 3-face, then $f$ is incident with two $9^+$-vertices by Proposition \ref{false-3-face}, which implies $c'(f)=3-4+2\times\frac{1}{2}=0$ by R2.

If $f$ is a non-special $4$-face, then no rule is valid for $f$ and thus $c'(f)= c(f)=0$. 
If $f$ is a special $4$-face, then  $c'(f)=4-4+1-1=0$ by R3.

If $f$ is a 5-face, then $f$ is incident with at most one 2-vertex by Proposition \ref{big-face}. Therefore, the remaining charge of $f$ after R1--R4 are applied to it is at least $5-4-1=0$, and thus $f$ has a nonnegative final charge by R5.
\end{proof}

\begin{clm}\label{clm:6-plus-face-nonneg}
Every $6$-face is incident with at most two special $2$-vertices in $\gx$.
Therefore, every $6^+$-face in $\gx$ has a nonnegative final charge. 
\end{clm} 

\begin{proof}
Suppose, to the contrary, that $f=uxvywz$ is a 6-face such that $u,v,w$ are special 2-vertices and $x,y,z$ are false vertices. According to the definition of the special 2-vertices, there are three $11^+$-vertices $u',v'$ and $w'$ such that $uv'$ (resp.\,$v'w$ and $u'w$) crosses $u'v$ (resp.\,$vw'$ and $uw'$) in $G$ at the crossing $x$ (resp.\,$y$ and $z$), see Figure \ref{6-face}(a). Pulling the vertex $v$ (resp.\,$w$) into the face of $\gx$ that is incident with the path $u'zw'$ (resp.\,$u'xv'$), we get another one 1-planar drawing of $G$ with three less crossings, see Figure \ref{6-face}(b). This contradicts the initial assumption that the drawing of $G$ has the minimum number of crossings. 

Therefore, every $6$-face in $\gx$ has charge at least $6-4-2\times 1=0$ after  R1--R4 are applied to it, and thus has nonnegative final charge by R5. 

On the other hand, every $d$-face $f$ with $d\geq 7$ is incident with at most $d/2$  2-vertices if $d$ is even, and at most $(d-3)/2$ 2-vertices if $d$ is odd, by Proposition \ref{big-face}. Therefore, after R1--R4 are applied to $f$, $f$ remains charge at least $d-4-d/2=(d-8)/2\geq 0$ if $d$ is even (i.e., $d\geq 8$), and at least $d-4-(d-3)/2=(d-5)/2\geq 1$ if $d$ is odd (i.e., $d\geq 7$). Hence $f$ has nonnegative final charge by R5.
\end{proof}

\begin{figure}
    \centering
    \includegraphics[width=15cm]{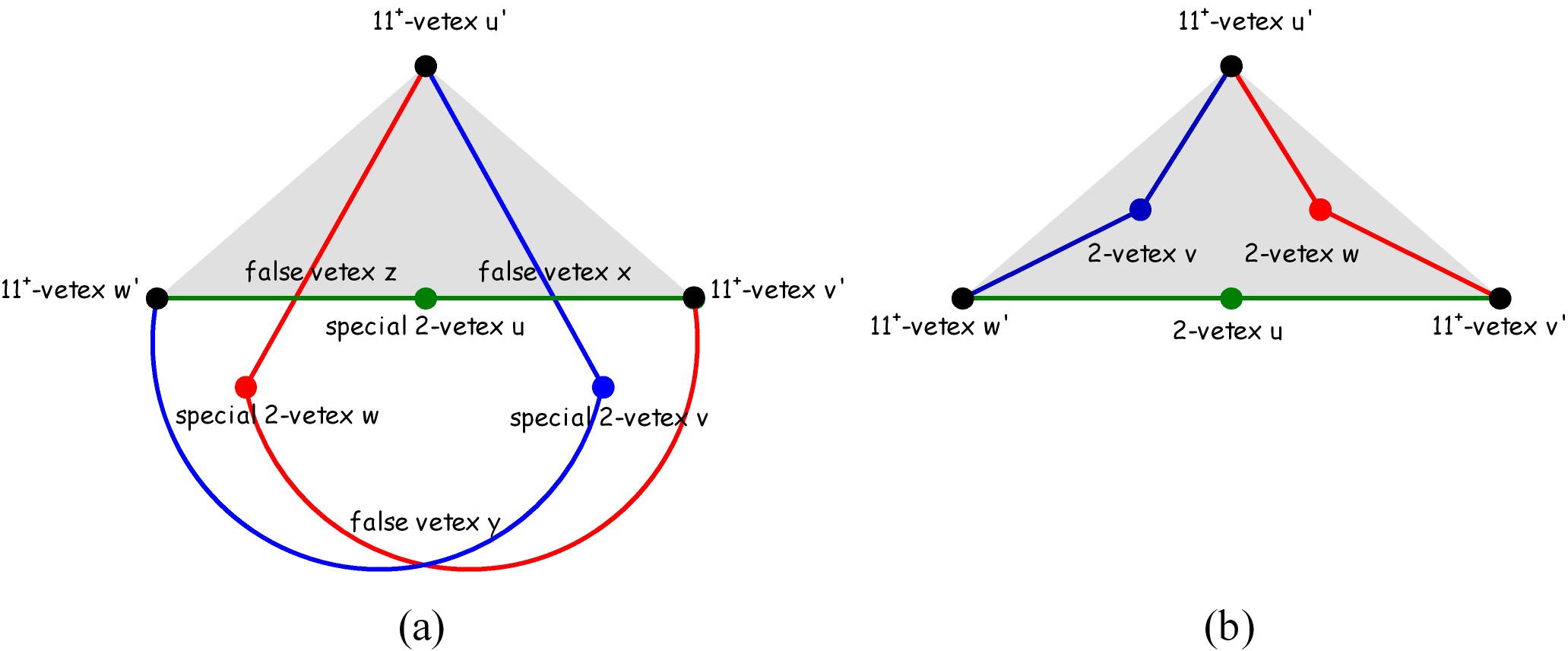}
    \caption{A redrawing with three less crossings: the proof of the first part of Claim \ref{clm:6-plus-face-nonneg}}
    \label{6-face}
\end{figure}

\begin{clm}\label{clm:5-plus-face-to-small-v}
Let $f$ be a $5^+$-face in $\gx$ with $u,x,v,y$ and $w$ being five consecutive vertices on the boundary of $f$ such that $u$ is a $11^+$-vertex, $v$ is a non-special $2$-vertex or a (special or non-special) $3$-vertex, and $x,y$ are false vertices.

$(1)$ If $w$ is a $11^+$-vertex, then $f$ sends at least $2$ to $v$;

$(2)$ If $w$ is a $10^-$-vertex, then $f$ sends at least $1$ to $v$.
\end{clm}

\begin{proof}
Let $a$ (resp.\,$b$) be the number of special $2$-vertices (resp.\,non-special $2$-vertices and special or non-special $3$-vertices) that are incident with $f$.

(1) Suppose that $w$ is a $11^+$-vertex. By Proposition \ref{big-face}, there are at least $a+(b-1)+1$ false vertices in $V_{\gx}(f)\backslash \{u,x,v,y,w\}$. This implies that 
$$a+(b-1)+a+(b-1)+1+5=2a+2b+4\leq d.$$
Therefore, $f$ sends to $v$ at least
\begin{align*}
    \frac{d-4-a}{b}\geq \frac{2a+2b+4-4-a}{b}\geq 2
\end{align*}
by R4 and R5.

(2) Suppose that $w$ is a $10^-$-vertex. By Proposition \ref{big-face}, there are at least $a+(b-1)$ false vertices in $V_{\gx}(f)\backslash \{u,x,v,y\}$. This implies that 
$$a+(b-1)+a+(b-1)+4=2a+2b+2\leq d.$$
Therefore, by R4 and R5, $f$ sends to $v$ at least
\begin{align*}
    \frac{d-4-a}{b}\geq \frac{2a+2b+2-4-a}{b}\geq 1
\end{align*}
if $a+b\geq 2$.

On the other hand, if $a+b\leq 1$, then $a=0$ and $b=1$, since $b\geq 1$. Hence $f$ would send at least $d-4\geq 1$ to $v$ by R5.
\end{proof}

\begin{clm}\label{clm:2-vertex-nonneg}
Every $2$-vertex in $\gx$ has a nonnegative final charge. 
\end{clm}

\begin{proof}
By Propositions \ref{true-3-face} (applying it by choosing $\mathcal{F}$ as $\mathcal{G}_1^-$) and \ref{false-3-face}, every $2$-vertex $v$ in $\gx$ is not incident with a 3-face in $\gx$. By Proposition \ref{big-face}, the neighbors of $v$ in $\gx$, say $x$ and $y$, are both false vertices. 

If $v$ is a special 2-vertex, then $v$ receives 1 from its incident special 4-face, say $uxvy$, by R3. If the other face incident with $v$ in $\gx$ is still a 4-face, say $wxvy$, then there would be two edges in $G$ connecting $u$ to $w$, one passing through the crossing $x$ and the other passing through the crossing $y$. This contradicts the fact that $G$ is a simple graph. Therefore,  $v$ is incident with a $5^+$-face, from which $v$ receives another 1 by R4. Hence $c'(v)=2-4+1+1=0$.

So in the following, we assume that $v$ is a non-special 2-vertex. 

If $v$ is incident with a 4-face, say $uxvy$, then $d_{\gx}(u)\leq 10$ since $v$ is non-special. Let $u_1$ (resp.\,$u_2$) be the vertices in $G$ such that $uu_1$ (resp.\,$uu_2$) passes through the crossing $x$ (resp.\,$y$). Since $G$ is a simple graph, $u_1\neq u_2$, and moreover, $u_1$ and $u_2$ are $11^+$-vertices by Proposition \ref{edge-with-big-vertex}. Therefore, $v$ is incident with a $5^+$-face that satisfies the condition of Claim \ref{clm:5-plus-face-to-small-v}(1). Since such a face would send at least 2 to $v$ by Claim \ref{clm:5-plus-face-to-small-v}(1),  $c'(v)\geq 2-4+2=0.$


If $v$ is incident with two $5^+$-faces $f_1$ and $f_2$, then let $u_1u_2$ and $w_1w_2$ be edges of $G$ that pass through the crossings $x$ and $y$, respectively, such that $u_i$ and $w_i$ are vertices on $f_i$, where $i=1,2$. By Proposition \ref{edge-with-big-vertex}, there are at least two $11^+$-vertices among $u_1,u_2,w_1$ and $w_2$. Therefore, either $f_1$ or $f_2$ satisfies the condition of Claim \ref{clm:5-plus-face-to-small-v}(1), or both $f_1$ and $f_2$ satisfy the condition of Claim \ref{clm:5-plus-face-to-small-v}(2). In each case $v$ receives at least 2 from $f_1$ and $f_2$, and thus $c'(v)\geq 2-4+2=0$.
\end{proof}

\begin{clm}\label{clm:3-vertex-nonneg}
Every $3$-vertex in $\gx$ has a nonnegative final charge. 
\end{clm}

\begin{proof}
By Propositions \ref{true-3-face} and \ref{false-3-face}, every $3$-vertex $v$ in $\gx$ is not incident with a 3-face in $\gx$. By Proposition \ref{big-face}, the three neighbors of $v$ in $\gx$, say $x,y$ and $z$, are false vertices. Let $f_1,f_2$ and $f_3$ be the face that is incident with the path $xvy$, $yvz$ and $zvx$ in $\gx$. 

Let $x_1x_3$ be the edge of $G$ that passes through the crossings $x$, where $x_1\in V_{\gx}(f_1)$ and $x_3\in V_{\gx}(f_3)$. By Proposition \ref{edge-with-big-vertex}, either $x_1$ or $x_3$, say $x_1$, is a $11^+$-vertex. If $f_1$ is a $5^+$-face, then it satisfies the condition of Claim \ref{clm:5-plus-face-to-small-v}(1) or Claim \ref{clm:5-plus-face-to-small-v}(2). This implies that $v$ receives at least 1 from $f_1$, and thus $c'(v)\geq 3-4+1=0$. Hence we assume that $f_1$ is a 4-face. Actually, $f_1$ is a special $4$-face now, from which $v$ receives $1$ by R3. This implies that $c'(v)\geq 3-4+1=0$.  
\end{proof}

\begin{clm}\label{clm:no-adj-special-4-faces}
No two special $4$-faces sharing a common $11^+$-vertex are adjacent in $\gx$.
\end{clm}

\begin{proof}
Suppose, to the contrary, $f_1=vxuy$ and $f_2=vywz$ are two adjacent special 4-faces in $\gx$ so that $d_{\gx}(v)\geq 11$. By the definition of the special $4$-face,  $u$ and $w$ are $3^-$-vertices and $y$ is a false vertex. This implies that $uw\in E(G)$, contradicting Proposition \ref{edge-with-big-vertex}.
\end{proof}

\begin{clm}\label{clm:giving-three-consecutive-faces}
If $v$ is a $11^+$-vertex and $f_1,f_2$ and $f_3$ are three consecutive faces that are incident with $v$ in $\gx$, then $v$ totally sends to $f_1, f_2$ and $f_3$ at most $2$.
\end{clm}

\begin{proof}
If there is only one special 4-face among $f_1,f_2$ and $f_3$, then by R1--R3, $v$ totally sends to $f_1, f_2$ and $f_3$ at most $1+\frac{1}{2}+\frac{1}{2}=2$. If there are at least two special 4-faces among $f_1,f_2$ and $f_3$, then by Claim \ref{clm:no-adj-special-4-faces}, they are $f_1$ and $f_3$, and $f_2$ is a non-special $4^+$-face. In this case  
$v$ totally sends $1+0+1=2$ to $f_1, f_2$ and $f_3$ by R3.
\end{proof}

\begin{clm}\label{clm:4-plus-vertex-nonneg}
Every $4^+$-vertex in $\gx$ has a nonnegative final charge. 
\end{clm}

\begin{proof}
Since vertices of degree between 4 and 8 are not involved in the discharging rules, their final charges are the same with their initial charges, which are nonnegative. Suppose that $v$ is a vertex of degree $d\geq 9$.

If $9\leq d\leq 10$, then $c'(v)\geq d-4-\frac{1}{2}d>0$ by R2.

If $d\geq 11$, then let $f_1,f_2,\ldots,f_d$ be the faces in this order around $v$. Let $\alpha_i$ with $1\leq i\leq d$ be the charge that $v$ sends to $f_i$ and let $\omega_i=\alpha_i+\alpha_{i+1}+\alpha_{i+2}$, where the subscripts are taken modular $d$. One can see that
$$\sum_{i=1}^d \alpha_i=\frac{1}{3} \sum_{i=1}^d \omega_i\leq \frac{2}{3}d,$$
where the second inequality holds by Claim \ref{clm:giving-three-consecutive-faces}. 
Hence $c'(v)=d-4-\sum_{i=1}^d \alpha_i\geq \frac{1}{3}d-4\geq 0$ if $d\geq 12$.

We now consider the case when $d=11$ more carefully. If $v$ is incident with at most three special 4-faces, then $c'(v)\geq 11-4-3\times 1-8\times \frac{1}{2}=0$ by R1--R3. So we assume that $v$ is incident with at least four special 4-faces. This implies that there is an integer
$1\leq i\leq d$ such that $f_i$ and $f_{i+2}$ are special 4-faces, where the subscripts are taken modular $d$. Assume, without loss of generality, that $i=1$. In this case, $f_2$ shall be a non-special $4^+$-face and therefore $\alpha_2=0$. By Claim \ref{clm:no-adj-special-4-faces}, $f_d$ and $f_4$ cannot be special 4-faces, to each of which $v$ sends at most $\frac{1}{2}$ by R1 and R2. This implies that $\omega_{11}\leq \frac{1}{2}+1+0=\frac{3}{2}$ and $\omega_2\leq 0+1+\frac{1}{2}=\frac{3}{2}$. Hence by Claim \ref{clm:giving-three-consecutive-faces}, we conclude that
$$\sum_{i=1}^{11} \alpha_i=\frac{1}{3} \bigg(\omega_2+\omega_{11}+\sum_{i\leq 10, i\neq 2} \omega_i\bigg)
\leq \frac{1}{3}\times \bigg(\frac{3}{2}+\frac{3}{2}+2\times 9\bigg)=7.$$
This implies that $c'(v)=11-4-\sum_{i=1}^{11} \alpha_i\geq 11-4-7=0$.
\end{proof}

\section{Remarks and Open Problems}

In this paper we have proved
\[
 7\leq \chi(\mathcal{G}_2)\leq \chi^d(\mathcal{G}_1)\leq \ch^d(\mathcal{G}_1)\leq 11.~~~~~~(\star)
\]
Hence a natural problem is to close the gap between the lower and the upper bounds in $(\star)$. In other words, we propose

\begin{pblm}\label{prob:1}
Determine the minimum integers $\ell_1$ and $\ell_2$ so that every 1-planar graph is dynamically $\ell_1$-colorable and  dynamically $\ell_2$-choosable, respectively.
\end{pblm}

On the other hand, one can see that the first relationship between the proper coloring of 2-planar graphs and the dynamic coloring of 1-planar graphs has been established  by Fact \ref{fact-1}. Actually, if we have a better lower bound for $\chi(\mathcal{G}_2)$, then we can improve 7 in $(\star)$ immediately. We think this may be a good motivation to study the proper coloring of 2-planar graphs. In view of this, we pose the following 

\begin{pblm}\label{prob:2}
Does there exist $2$-planar graph with chromatic number $8$ or $9$  ?
\end{pblm}

Note that $K_8$ is not 2-planar, which was very recently proved by Angelini, Bekos, Kaufmann and Schneck \cite{angelini2019efficient}. 
On the other hand, Dmitry Karpov (personal communication) announced a proof of 9-colorability of 2-planar graphs (written in Russian). 

\section*{Acknowledgements}

The first author would like to thank Fedor Petrov who reminded him at MathOverflow \cite{ZHANG} on an unpublished result of Dmitry Karpov that $\chi(\mathcal{G}_2)\leq 9$, and appreciate Dmitry Karpov for the personal communication with him on this topic. The supports provided by China Scholarship Council (CSC) and Institute for Basic Science (IBS, Korea) during a visit of the first author to Discrete Mathematics Group, IBS are acknowledged.

\bibliographystyle{abbrv}
\bibliography{dynamic-coloring,1-planar-graphs,else}

\end{document}